\documentclass[11pt]{article}
    
\usepackage{amsthm}
\usepackage{amssymb}
\usepackage{hyperref}                           
\usepackage{mathrsfs}
\usepackage{color} 
\usepackage{amsmath}                     
\usepackage{diagbox}      
\usepackage[normalem]{ulem}         

\def\fl#1{\left\lfloor#1\right\rfloor}

\def\sttf2#1#2{\left[\!\!\left[#1\atop#2\right]\!\!\right]}
\def\stf3f#1#2{\left[\!\!\left[\!\!\left[#1\atop#2\right]\!\!\right]\!\!\right]}
\def\stff4#1#2{\left[\!\!\left[\!\!\left[\!\!\left[#1\atop#2\right]\!\!\right]\!\!\right]\!\!\right]}
\def\stss2#1#2{\left\{\!\!\left\{#1\atop#2\right\}\!\!\right\}}

\newtheorem{theorem}{Theorem}
\newtheorem{Prop}{Proposition}
\newtheorem{Cor}{Corollary}
\newtheorem{Lem}{Lemma}

\allowdisplaybreaks 

\title{Continued fractions, determinant expressions, and identities}  
\author{
Nikita Kalinin\\
Mathematics and Computer Science Department\\
Guangdong Technion Israel Institute of Technology\\
Shantou 515603, China\\
\texttt{nikita.kalinin@gtiit.edu.cn}
\and
Takao Komatsu\\
Institute of Mathematics, Henan Academy of Sciences\\
Zhengzhou 450046, China\\
\and
Department of Mathematics, Institute of Science Tokyo\\
2-12-1 Ookayama, Meguro-ku, Tokyo 152-8551, Japan\\
\texttt{komatsu@zstu.edu.cn}\\
\texttt{komatsu.t.al@m.titech.ac.jp}
}
\date{}

\begin{document}

\maketitle

\begin{abstract}
In this paper, we clarified the relationship between continued fractions, determinants, and identities, making it easier to apply these methods systematically in other settings. In particular, we studied finite continued fractions from the perspective of incomplete numbers (restricted or associated numbers) and also explored their relationships with determinant representations and identities. Most of the new results in this paper concern 
$q$-analogues of special numbers, whereas the classical cases mainly serve to illustrate and unify the general framework. The framework developed here is flexible and allows one to derive continued fractions, determinant formulas, and coefficient identities in a uniform way for several new 
$q$-families, and it is expected to be applicable to other families of special numbers, as well.

{\small keywords: determinants; continued fractions; identities; $q$-analogues} 

\end{abstract}

\section{Introduction}  

The continued fraction expansions of generating functions, functions, and series are not unique, unlike the continued fractions of real numbers, which makes them rich but difficult to systematize.  
One well-known example is the infinite series $\sum_{n=1}^\infty 2^{-\fl{n\alpha}}$ with $\alpha=(1+\sqrt{5})/2$, which can be expressed as a continued fraction whose convergents involve powers of two with Fibonacci exponents \cite{Davison}. This result has been generalized a little \cite{Bundschuh}, but has not been extended systematically.   
Another famous example concerns the Bernoulli numbers $B_n$, defined by the generating function 
\begin{equation}  
\sum_{n=0}^\infty B_n\frac{x^n}{n!}=\frac{x}{e^x-1}\,.  
\label{def_ber}
\end{equation}  
By using their continued fractions, Kaneko \cite{Kaneko95} obtained the recurrence relations.  
Frame \cite{Frame} gave continued fraction expansions of the divergent Bernoulli number series $\sum B_{2 n}x^{2 n}$ and some related series. More continued fraction expansions of Bernoulli numbers are given by D. Zagier in \cite[A.5]{AIK}. However, no concrete method for obtaining continued fractions has been found.  Some analytical theories can be found in \cite{JT,Wall}.  

On the other hand, it is interesting to attempt to express numbers themselves using determinants. This may be regarded as a special case of the so-called Hankel determinant, but historically, Glaisher's results \cite{Glaisher} are classically famous. He expressed the Bernoulli numbers as 
$$
B_n=(-1)^n n!\left|\begin{array}{ccccc}  
\frac{1}{2!}&1&0&&\\ 
\frac{1}{3!}&\frac{1}{2!}&1&&\\ 
\vdots&&\ddots&&0\\ 
\frac{1}{n!}&\frac{1}{(n-1)!}&&\ddots&1\\
\frac{1}{(n+1)!}&\frac{1}{n!}&\cdots&\frac{1}{3!}&\frac{1}{2!}
\end{array}\right| 
$$
and more in various forms, and also provided determinant representations for Cauchy numbers $c_n$ (also known as Gregory numbers or Bernoulli numbers of the second kind), Euler numbers $E_n$, and others. Such forms of determinants were studied by Trudi \cite{Trudi} and by Brioschi \cite{Brioschi} before Glaisher. Further theory and historical background can be found in \cite{Muir}.  
These determinantal expressions have been generalized (see, e.g., \cite{KKL22,Ko20b,KP19,KR18,KZ}) and applied to the finite (incomplete) cases (see, e.g., \cite{Ko16,Ko18,KLM16,KMS16}).  

Classically, continuant theory starts from a prescribed continued fraction and derives determinant formulas for its convergents (see, e.g., \cite{Perron,Wall}). Likewise, Trudi- or Jacobi--Trudi-type formulas provide determinant expressions once the relevant coefficient sequence is fixed (see, e.g., \cite{Muir,Trudi}). Our point of view in this paper is different. Rather than treating continued fractions and determinant formulas as separate outputs, we regard continued fractions, reciprocal coefficient identities, and Toeplitz--Hessenberg determinant expressions as different manifestations of the same mechanism. We make this mechanism explicit in a form that can be transferred systematically from one family of special numbers to another, and we also develop its finite/incomplete version, where truncation on the continued-fraction side naturally corresponds to restricted or associated numbers on the coefficient~side.

The purpose of this paper is to formulate this unified framework explicitly and to show how it can be applied systematically. After presenting the general correspondence between continued fractions, determinant expressions, and coefficient identities, we develop its finite/incomplete analogue and apply it to several families of special numbers, including $q$-Bernoulli numbers, $q$-Cauchy numbers, $q$-hypergeometric Euler numbers, Lehmer--Euler numbers, and generalized harmonic numbers. The same viewpoint is expected to be useful in other contexts, as well; see, for example, \cite{Ko26a}.

{Section 3 applies this framework to 
$q$-Bernoulli, 
$q$-Cauchy, and 
$q$-hypergeometric Euler-type numbers, where the resulting statements are new 
$q$-analogues of known classical formulas. Section~4 treats the finite/incomplete setting, and Sections 5 and 6 discuss applications to Lehmer–Euler and hyperharmonic-type numbers.}

{The main contributions of the paper are therefore: a unified framework theorem, its systematic application to new 
$q$-families, and its finite/incomplete analogue.}

\section{Unified framework}  

Several types of continued fractions of the (generating) functions have been introduced and studied. For example, the generating function of Cauchy numbers $c_n$ has the continued fraction    
\begin{align}  
&\sum_{n=0}^\infty c_n\frac{x^n}{n!}=\frac{x}{\log(1+x)}\notag\\
&=1+\cfrac{x}{2-x+\cfrac{2^2 x}{3-2 x+\cfrac{3^2 x}{4-3 x+\ddots}}}
\label{j-cont-cauchy} 
\end{align}   
\cite[Chapter 8]{Loya} and another continued fraction 
\begin{align}  
&\sum_{n=0}^\infty c_n\frac{x^n}{n!}=\frac{x}{\log(1+x)}\notag\\
&=1+\cfrac{1^2 x}{2+\cfrac{1^2 x}{3+\cfrac{2^2 x}{4+\cfrac{2^2 x}{5+\cfrac{3^2 x}{6+\cdots}}}}}
\label{t-cont-cauchy} 
\end{align}   
({\it cf.}\, \cite[(90.1)]{Wall}). 
The former type of continued fraction expansion, as in (\ref{j-cont-cauchy}), has been often studied. A similar but more general type of continued fraction is the Jacobi-type continued fraction ($J$-fraction), which is usually written in the form
$$
\cfrac{1}{1-c_0 z-\cfrac{b_1}{1-c_1 z-\cfrac{b_2}{1-c_2 z-{\atop\ddots}}}}\,  
$$ 
One of the advantages of such fractions is to yield the continued fraction expansions and determinant expressions more easily.  
In \cite{Ko21}, several continued fraction expansions and determinant expressions of hypergeometric Cauchy numbers, shifted Cauchy numbers, and leaping Cauchy numbers are given.  
Some special types are known as $J$-fractions, $C$-fractions, $T$-fractions, $M$-fractions, and Hankel continued fractions (see, e.g., \cite{Frame,JT,Perron,Wall}).

For a certain class of $T$-fractions, there are useful relations yielding determinant expressions and further identities.  The next result combines familiar ingredients from continued-fraction theory, inversion relations, and Trudi-type determinant identities, and arranges them in a single framework convenient for the applications developed below. {Its significance is that the same formal mechanism simultaneously produces continued fractions, determinant expressions, and reciprocal coefficient identities, and can then be transferred systematically to new $q$-families and truncated settings.}

\begin{theorem}  
For three sequences $\{f_n\}_{n\ge 0}$ (with $f_0=1$), $\{g_n\}_{n\ge 1}$, and $\{h_n\}_{n\ge 1}$, we have (assuming $\frac{H_0}{G_0}=1$)
\begin{align*}  
\sum_{n=0}^\infty f_n x^n&=\left(\sum_{j=0}^\infty\frac{h_1\cdots h_j}{g_1\cdots g_j}x^j\right)^{-1}:=\left(\sum_{j=0}^\infty\frac{H_j}{G_j}x^j\right)^{-1}\\
&=1-\cfrac{h_1 x}{g_1+h_1 x-\cfrac{g_1 h_2 x}{g_2+h_2 x-\cfrac{g_2h_3 x}{g_3+h_3 x-\ddots}}}\\
&\Longleftrightarrow\\ 
&f_n=(-1)^n\left|\begin{array}{ccccc}
\frac{H_1}{G_1}&1&0&&\\  
\frac{H_2}{G_2}&\frac{H_1}{G_1}&1&&\\
\vdots&\vdots&\ddots&&0\\
\frac{H_{n-1}}{G_{n-1}}&\frac{H_{n-2}}{G_{n-2}}&\cdots&\frac{H_1}{G_1}&1\\ 
\frac{H_n}{G_n}&\frac{H_{n-1}}{G_{n-1}}&\cdots&\frac{H_2}{G_2}&\frac{H_1}{G_1}
\end{array}\right|\\
&\Longleftrightarrow\quad\sum_{k=0}^n\frac{f_k H_{n-k}}{G_{n-k}}=\begin{cases}
1&\text{if $n=0$},\\
0&\text{if $n\ge 1$,}
\end{cases}\\ 
&\Longleftrightarrow\\ 
&\frac{H_n}{G_n}=(-1)^n\left|\begin{array}{ccccc}
f_1&1&0&&\\  
f_2&f_1&1&&\\
\vdots&\vdots&\ddots&&0\\
f_{n-1}&f_{n-2}&\cdots&f_1&1\\ 
f_n&f_{n-1}&\cdots&f_2&f_1
\end{array}\right|\\
&\Longleftrightarrow\\ 
&f_n=\sum_{t_1+2 t_2+\cdots+n t_n=n}\binom{t_1+\cdots+t_n}{t_1,\dots,t_n}(-1)^{t_1+\cdots+t_n}\left(\frac{H_1}{G_1}\right)^{t_1}\left(\frac{H_2}{G_2}\right)^{t_2}\cdots\left(\frac{H_n}{G_n}\right)^{t_n}\\
&\phantom{f_n}=\sum_{k=1}^n(-1)^k\sum_{i_1+\cdots+i_k=n\atop i_1,\dots,i_k\ge 1}\frac{H_{i_1}}{G_{i_1}}\cdots\frac{H_{i_k}}{G_{i_k}}\\
&\Longleftrightarrow\\ 
&\frac{H_n}{G_n}=\sum_{t_1+2 t_2+\cdots+n t_n=n}\binom{t_1+\cdots+t_n}{t_1,\dots,t_n}(-1)^{t_1+\cdots+t_n}f_1^{t_1}f_2^{t_2}\cdots f_n^{t_n}\\
&\phantom{\frac{H_n}{G_n}}=\sum_{k=1}^n(-1)^k\sum_{i_1+\cdots+i_k=n\atop i_1,\dots,i_k\ge 1}f_{i_1}\cdots f_{i_k}\,.   
\end{align*}
\label{th:det-cont}
\end{theorem} 
\begin{proof}  We begin by collecting several ingredients from the literature into a single framework that will serve as the basis for the finite/incomplete extensions and the applications developed later.
The first part comes from \cite[p.~696]{Ko20}. The $n$-th convergent $P_n(x)/Q_n(x)$ is given by 
$$
P_n(x)=g_1\cdots g_n\quad\hbox{and}\quad Q_n(x)=g_1\cdots g_n\sum_{j=0}^n\frac{h_1\cdots h_j}{g_1\cdots g_j}x^j\,, 
$$ 
respectively, satisfying the recurrence relations 
\begin{align*}
P_n(x)&=(g_n+h_n x)P_{n-1}(x)-g_{n-1}h_n x P_{n-2}(x)\quad(n\ge {2})\,,\\
Q_n(x)&=(g_n+h_n x)Q_{n-1}(x)-g_{n-1}h_n x Q_{n-2}(x)\quad(n\ge {2})
\end{align*} 
with $P_{-1}(x)=P_0(x)=Q_0(x)=1$ and $Q_{-1}(x)=0$.  Hence, 
$$
\frac{P_n(x)}{Q_n(x)}\equiv \sum_{k=0}^\infty f_k x^k\pmod{x^{k+1}}
$$ 
or 
$$
\sum_{n=0}^\infty f_n x^n=\lim_{n\to\infty}\frac{P_n(x)}{Q_n(x)}=\left(\sum_{j=0}^\infty\frac{h_1\cdots h_j}{g_1\cdots g_j}x^j\right)^{-1}\,. 
$$ 
See also \cite{Ko21}.  

The second part is due to Cameron's operator \cite{Cameron} for one sequence $\{f_n\}_{n\ge 0}$ to another sequence $\{h_n/g_n\}_{n\ge 0}$. As a result of this, it is possible to apply it to many special numbers. More details can be seen in \cite{KKL22}. Very recently, by combining it with the application of the method used in Hessenberg matrices \cite{Kilic10,KT10,KTH10, BelbachirBelkhirDjellas22, GoyShattuck24, GoyShattuck24Leonardo,GokcanDeger25}, it is possible to extend matrix identities for Fibonacci numbers to the setting of Fibonacci polynomials \cite{Ko25a,KK26}.

The third part can be obtained by applying the inversion relation (see, e.g., \cite{KR18}):  
\begin{align*}
&\sum_{k=0}^n(-1)^{n-k}\alpha_k R(n-k)=0\quad(n\ge 1)\\
&\Longleftrightarrow\\  
&\alpha_n=\begin{vmatrix} R(1) & 1 & & 0\\
R(2) & \ddots &  \ddots & \\
\vdots & \ddots &  \ddots & 1\\
R(n) & \cdots &  R(2) & R(1) \\
 \end{vmatrix}\quad 
\Longleftrightarrow\quad 
&R(n)=\begin{vmatrix} \alpha_1 & 1 & &0 \\
\alpha_2 & \ddots &  \ddots & \\
\vdots & \ddots &  \ddots & 1\\
\alpha_n & \cdots &  \alpha_2 & \alpha_1 \\
 \end{vmatrix}\,.
\end{align*}  

The remaining equivalences follow from Trudi’s formula \cite[Vol.~3, p.~214]{Muir},\cite{Trudi}:
\begin{align*}
&\begin{vmatrix} a_1 & a_0 & &0 \\
a_2 & \ddots &  \ddots & \\
\vdots & \ddots &  \ddots & a_0\\
a_n & \cdots &  a_2 & a_1 \\
 \end{vmatrix}\\
&=\sum_{t_1+2 t_2+\cdots+n t_n=n}\binom{t_1+\cdots+t_n}{t_1,\dots,t_n}(-a_0)^{n-t_1-\cdots-t_n}a_1^{t_1}a_2^{t_2}\cdots a_n^{t_n}\,, 
\end{align*} 
where $\binom{t_1+\cdots+t_n}{t_1,\dots,t_n}=\frac{(t_1+\cdots+t_n)!}{t_1!\cdots t_n!}$ is the multinomial coefficient. 
The double summation parts are due to \cite[Theorem 1]{Ko20b} with $m\to\infty$ or \cite[Theorem 3]{Ko20b} with $m=1$.   
\end{proof}

\section{Applications to $q$-Bernoulli and related numbers}   

The so-called $q$-numbers $[n]_q$ are defined by 
$$
[n]_q:=\frac{1-q^n}{1-q}\quad(q\ne 1)\,. 
$$

Then the $q$-factorial is given by $[n]_q!=[n]_q[n-1]_q\cdots[2]_q[1]_q$ ($n\ge 1$) with $[0]_q!=1$.  
The (Gauss) $q$-binomial coefficients are given by 
$$
\binom{n}{k}_q:=\frac{[n]_q!}{[k]_q![n-k]_q!}\,. 
$$   

Observe that
\[
[n]_q=1+q+\cdots+q^{n-1}\to n \qquad (q\to 1),
\]
so \(q\)-factorials tend to the usual factorials. In particular, for the two
standard \(q\)-exponentials
\[
e_q(x)=\sum_{n\ge0}\frac{x^n}{(q;q)_n},
\qquad
E_q(x)=\sum_{n\ge0}\frac{q^{n(n-1)/2}x^n}{(q;q)_n},
\]
one has, since \((q;q)_n=(1-q)^n[n]_q!\),
\[
e_q((1-q)x)\to e^x,
\qquad
E_q((1-q)x)\to e^x
\qquad (q\to 1).
\]
So \(e_q\) and \(E_q\) are distinct \(q\)-analogues away from \(q=1\), but both
recover the ordinary exponential in the classical limit.

There are many $q$-Bernoulli numbers. For example, Carlitz \cite{Carlitz48} defined them as 
$$
\beta_n:=\frac{1}{(1-q)^n}\sum_{k=0}^n(-1)^k\binom{n}{k}\frac{k+1}{[k+1]_q}
$$ 
or 
$$
\sum_{k=0}^n\binom{n}{k}q^{k+1}\beta_k-\beta_n=\begin{cases}  
1&\text{if $n=1$},\\
0&\text{if $n\ge 2$}\,. 
\end{cases}
$$ 
Here, $\binom{n}{k}$ are the usual binomial coefficients.   
However, these $q$-Bernoulli numbers do not have a beautiful determinant expression and/or a continued fraction expansion even though they satisfy other elegant relations.

\subsection{$q$-Bernoulli numbers}  

For simplicity, we write $[n]:=[n]_q$ without $q$. 
Now,  a $q$-Bernoulli number $\mathcal B_n=\mathcal B_{n,q}$ is defined by the generating function 
\begin{equation}  
\sum_{k=0}^\infty\mathcal B_k\frac{x^k}{k!}=\frac{x}{E_q(x)-1}\,,  
\label{def:q-ber}
\end{equation}
where 
$$
E_q(z):=\sum_{n=0}^\infty\frac{z^n}{[n]!}
$$
is a $q$-exponential function.  Note that another $q$-exponential function is defined by 
$$
e_q(z):=\sum_{n=0}^\infty\frac{z^n}{(1-q)(1-q^2)\cdots(1-q^n)}
$$
(see, e.g., \cite{Ernst}).  
In fact, more generalizations of $q$-Bernoulli numbers can be defined ({\it cf.\,\,}\cite[Theorem 1]{Ko20}), but for simplicity, we give the results about $q$-hypergeometric Bernoulli numbers $\mathcal B_{N,n}=\mathcal B_{N,n,q}$ by applying Theorem \ref{th:det-cont} with $g_j=[N+j]_q$ and $h_j=1$ ($N\ge 1$,\,$j\ge 1$). When $N=1$, $\mathcal B_n=\mathcal B_{1,n}$ are $q$-Bernoulli numbers defined above.  
 
Although the results follow immediately from Theorem \ref{th:det-cont}, in this subsection, we shall see the details independently to understand the structures more clearly.

The $q$-hypergeometric Bernoulli numbers $\mathcal B_{N,n}$ have an explicit continued fraction expansion.   

\begin{theorem}  
We have 
\begin{align}
&\sum_{k=0}^\infty\mathcal B_{N,k}\frac{x^k}{k!}:=\left(\sum_{j=0}^\infty\frac{[N]_q!}{[N+j]_q!}x^j\right)^{-1}\notag\\
&=1-\cfrac{x}{[N+1]_q+x-\cfrac{[N+1]_q x}{[N+2]_q+x-\cfrac{[N+2]_q x}{[N+3]_q+x-\ddots}}}\,
\label{eq:q-ber-cont}
\end{align} 
\label{th:q-ber-cont}
\end{theorem}

\noindent 
{\it Remark.}  
When $N=1$ and $q\to 1$ in Theorem \ref{th:q-ber-cont}, we get the continued fraction expansion of the generating function of the classical Bernoulli numbers (see, e.g., \cite[Corollary 2]{Ko20}).  
$$
\frac{x}{e^x-1}=\sum_{k=0}^\infty B_k\frac{x^k}{k!}=1-\cfrac{x}{2+x-\cfrac{2 x}{3+x-\cfrac{3 x}{4+x-\ddots}}}\, 
$$ 

\begin{proof}[Proof of Theorem \ref{th:q-ber-cont}.]
For $n\ge 0$, set 
$$
P_n(x):=\frac{[N+n]_q!}{[N]_q!}\quad\hbox{and}\quad Q_n(x):=[N+n]_q!\sum_{k=0}^n\frac{x^k}{[N+k]_q!}\,.
$$
Note that $Q_n(x)$ are the polynomials with integer coefficients. Then, we can see that $P_n(x)$'s and $Q_n(x)$'s satisfy the recurrence relations 
\begin{align*}
&P_n(x)=([N+n]_q+x)P_{n-1}(x)-[N+n-1]_q x P_{n-2}(x)\quad(n\ge 2),\\ 
&Q_n(x)=([N+n]_q+x)Q_{n-1}(x)-[N+n-1]_q x Q_{n-2}(x)\quad(n\ge 2),
\end{align*}  
with $P_0(x)=Q_0(x)=1$, $P_1(x)=[N+1]_q$, and $Q_1(x)=[N+1]_q+x$.  
Therefore, we know that $P_n(x)/Q_n(x)$ is the $n$-th convergent of the continued fraction expansion of the right-hand side of (\ref{eq:q-ber-cont}). Taking $n\to\infty$ for $P_n(x)/Q_n(x)$, we obtain the expression of the right-hand side of (\ref{eq:q-ber-cont}).  
On the other hand, by the approximation property in \cite[p.~696]{Ko20}, we have 
$$
\frac{P_n}{Q_n}\equiv
\left(\sum_{j=0}^\infty\frac{[N]_q!}{[N+j]_q!}x^j\right)^{-1}\pmod{x^{n+1}}\,. 
$$ 
Hence, taking $n\to\infty$ for $P_n(x)/Q_n(x)$, we obtain the expression of the left-hand side of (\ref{eq:q-ber-cont}).  
\end{proof}

The $q$-hypergeometric Bernoulli numbers $\mathcal B_{N,n}$ have a determinant expression.  

\begin{theorem}  
For $N,n\ge 1$, we have 
$$
\mathcal B_{N,n}=(-1)^n n!\left|\begin{array}{ccccc}
\frac{[N]_q!}{[N+1]_q!}&1&0&&\\  
\frac{[N]_q!}{[N+2]_q!}&\frac{[N]_q!}{[N+1]_q!}&1&&\\
\vdots&\vdots&\ddots&1&0\\
\frac{[N]_q!}{[N+n-1]_q!}&\frac{[N]_q!}{[N+n-2]_q!}&\cdots&\frac{[N]_q!}{[N+1]_q!}&1\\ 
\frac{[N]_q!}{[N+n]_q!}&\frac{[N]_q!}{[N+n-1]_q!}&\cdots&\frac{[N]_q!}{[N+2]_q!}&\frac{[N]_q!}{[N+1]_q!}
\end{array}\right|\,.  
$$
\label{th-det-mod-ber}
\end{theorem}

We need the following relation.   

\begin{Lem}  
We have 
$$
\sum_{k=0}^n\frac{[N]_q!\mathcal B_{N,k}}{k![N+n-k]_q!}=\begin{cases}
0&\text{if $n\ge 1$},\\
1&\text{if $n=0$}\,. 
\end{cases}
$$
\label{rel-mod-ber}
\end{Lem} 
\begin{proof}  
By the definition, 
\begin{align*}
1&=\left(\sum_{k=0}^\infty\mathcal B_{N,k}\frac{x^k}{k!}\right)\left(\sum_{j=0}^\infty\frac{[N]_q!}{[N+j]_q!}x^j\right)\\ 
&=\sum_{n=0}^\infty\left(\sum_{k=0}^n\frac{\mathcal B_{N,k}}{k!}\frac{[N]_q!}{[N+n-k]_q!}\right)x^n\,. 
\end{align*}
Comparing the coefficients on both sides, we get the desired result. 
\end{proof} 

\begin{proof}[Proof of Theorem \ref{th-det-mod-ber}.]  
For simplicity, set  
\begin{equation}
\mathcal B_n':=\left|\begin{array}{ccccc}
\frac{[N]_q!}{[N+1]_q!}&1&0&&\\  
\frac{[N]_q!}{[N+2]_q!}&\frac{[N]_q!}{[N+1]_q!}&1&&\\
\vdots&\vdots&\ddots&1&0\\
\frac{[N]_q!}{[N+n-1]_q!}&\frac{[N]_q!}{[N+n-2]_q!}&\cdots&\frac{[N]_q!}{[N+1]_q!}&1\\ 
\frac{[N]_q!}{[N+n]_q!}&\frac{[N]_q!}{[N+n-1]_q!}&\cdots&\frac{[N]_q!}{[N+2]_q!}&\frac{[N]_q!}{[N+1]_q!}
\end{array}\right|\,. 
\label{det-bdash}
\end{equation} 
By the definition, 
\begin{multline*}
\left(\sum_{n=0}^\infty\frac{[N]_q!}{[N+n]_q!}x^n\right)^{-1}=1-\frac{[N]_q!}{[N+1]_q!}x+\left(\frac{[N]_q!^2}{[N+1]_q!^2}-\frac{[N]_q!}{[N+2]_q!}\right)x^2-\cdots\,. 
\end{multline*} 
Hence, $\mathcal B_1'=\frac{[N]_q!}{[N+1]_q!}$, so the identity in Theorem \ref{th-det-mod-ber} is valid for $n=1$ (we can also see that it is valid for $n=2$).   
Assume that the identity is valid up to a certain $n-1\ge 0$.  Then, expanding the determinant along the first row repeatedly, the right-hand side of (\ref{det-bdash}) is equal to  
\begin{align*}  
&\frac{\mathcal B_{n-1}'[N]_q!}{[N+1]_q!}-
\left|\begin{array}{ccccc}
\frac{[N]_q!}{[N+2]_q!}&1&0&&\\ 
\frac{[N]_q!}{[N+3]_q!}&\frac{[N]_q!}{[N+1]_q!}&&&\\  
\vdots&&\ddots&&0\\
\frac{[N]_q!}{[N+n-1]_q!}&&&\frac{[N]_q!}{[N+1]_q!}&1\\ 
\frac{[N]_q!}{[N+n]_q!}&\frac{[N]_q!}{[N+n-1]_q!}&\cdots&\cdots&\frac{[N]_q!}{[N+1]_q!}  
\end{array}\right|\\
&=\frac{\mathcal B_{n-1}'[N]_q!}{[N+1]_q!}-\frac{\mathcal B_{n-2}'[N]_q!}{[N+2]_q!}+\left|\begin{array}{ccccc}
\frac{[N]_q!}{[N+3]_q!}&1&0&&\\ 
\frac{[N]_q!}{[N+4]_q!}&\frac{[N]_q!}{[N+1]_q!}&&&\\  
\vdots&&\ddots&&0\\
\frac{[N]_q!}{[N+n-1]_q!}&&&\frac{1}{[N+1]_q!}&1\\ 
\frac{[N]_q!}{[N+n]_q!}&\frac{[N]_q!}{[N+n-1]_q!}&\cdots&\cdots&\frac{[N]_q!}{[N+1]_q!}  
\end{array}\right|\\
&=\cdots\\ 
&=\frac{\mathcal B_{n-1}'[N]_q!}{[N+1]_q!}-\frac{\mathcal B_{n-2}'[N]_q!}{[N+2]_q!}+\cdots
 +(-1)^{n}\left|\begin{array}{cc}
\frac{[N]_q!}{[N+n-1]_q!}&1\\
\frac{[N]_q!}{[N+n]_q!}&\frac{[N]_q!}{[N+1]_q!} 
\end{array}\right|\\
&=\sum_{k=0}^{n-1}\frac{(-1)^{n-k-1}\mathcal B_k'[N]_q!}{[N+n-k]_q!}=\mathcal B_n'\,.
\end{align*} 
In the last part, by Lemma \ref{rel-mod-ber}, we used the relation 
$$
\sum_{k=0}^n\frac{(-1)^{k-1}\mathcal B_k'[N]_q!}{[N+n-k]_q!}=\begin{cases}
0&\text{if $n\ge 1$},\\
1&\text{if $n=0$}\,. 
\end{cases} 
$$
\end{proof}

\subsection{$q$-Cauchy numbers}  

The $q$-logarithm function is defined by 
$$
{\rm Log}_q(1+z):=\sum_{n=1}^\infty\frac{(-1)^{n-1}z^n}{[n]_q}\,.  
$$
Then, the $q$-Cauchy numbers $\mathcal C_n=\mathcal C_{n,q}$ are introduced via the generating function 
\begin{equation}  
\sum_{k=0}^\infty\mathcal C_k\frac{x^k}{k!}=\frac{x}{{\rm Log}_q(1+x)}\,.  
\label{def:q-cau}
\end{equation}
More generally, {we consider} the $q$-hypergeometric Cauchy numbers $\mathcal C_{N,n}=\mathcal C_{N,n,q}$, 
which recover the $q$-Cauchy numbers $\mathcal C_n=\mathcal C_{1,n}$ 
when $N=1$.  {Thus the theorem above may be viewed as a 
$q$-generalization of the classical result.}
By applying Theorem \ref{th:det-cont} with $g_j=[N+j]_q$ and $h_j=-[N+j-1]_q$ ($j\ge 1$), we can get the results about $q$-Cauchy numbers.    

\begin{theorem}  
For $n\ge 1$, we have 
$$
\mathcal C_{N,n}=n!\left|\begin{array}{ccccc}
\frac{[N]_q}{[N+1]_q}&1&0&&\\  
\frac{[N]_q}{[N+2]_q}&\frac{[N]_q}{[N+1]_q}&1&&\\
\vdots&\vdots&\ddots&1&0\\
\frac{[N]_q}{[N+n-1]_q}&\frac{[N]_q}{[N+n-2]_q}&\cdots&\frac{[N]_q}{[N+1]_q}&1\\ 
\frac{[N]_q}{[N+n]_q}&\frac{[N]_q}{[N+n-1]_q}&\cdots&\frac{[N]_q}{[N+2]_q}&\frac{[N]_q}{[N+1]_q}
\end{array}\right|\,.  
$$
\label{th-det-cau}
\end{theorem}
\begin{proof}
By Theorem \ref{th:det-cont}, we have 
{\small  
$$
\mathcal C_{N,n}=n!(-1)^n\left|\begin{array}{ccccc}
-\frac{[N]_q}{[N+1]_q}&1&0&&\\  
\frac{[N]_q}{[N+2]_q}&-\frac{[N]_q}{[N+1]_q}&1&&\\
\vdots&\vdots&\ddots&1&0\\
(-1)^{n-1}\frac{[N]_q}{[N+n-1]_q}&(-1)^{n-2}\frac{[N]_q}{[N+n-2]_q}&\cdots&-\frac{[N]_q}{[N+1]_q}&1\\ 
(-1)^n\frac{[N]_q}{[N+n]_q}&(-1)^{n-1}\frac{[N]_q}{[N+n-1]_q}&\cdots&\frac{[N]_q}{[N+2]_q}&-\frac{[N]_q}{[N+1]_q}
\end{array}\right|\,.  
$$ 
} 
Factoring the alternating signs from rows and columns removes the factor $(-1)^n$, and we get the desired result.  
\end{proof}

\begin{Lem}  
We have 
$$
\sum_{k=0}^n(-1)^{n-k}\frac{\mathcal C_{N,k}}{k![N+n-k]_q}=\begin{cases}
0&\text{if $n\ge 1$},\\
1&\text{if $n=0$}\,. 
\end{cases}
$$
\label{rel-cau}
\end{Lem}

The $q$-hypergeometric Cauchy numbers $\mathcal C_{N,n}$ have an explicit continued fraction expansion.   

\begin{theorem}  
We have 
\begin{multline}
\sum_{k=0}^\infty\mathcal C_{N,k}\frac{x^k}{k!}
=1+\cfrac{[N]_q x}{[N+1]_q-[N]_q x+\cfrac{[N+1]_q^2 x}{[N+2]_q-[N+1]_q x+\cfrac{[N+2]_q^2 x}{[N+3]_q-[N+2]_q x-\ddots}}}\, 
\label{eq:q-cau-cont}
\end{multline} 
\label{th:q-cau-cont}
\end{theorem}

\noindent 
{\it Remark.}  
When $q\to 1$ and $N=1$ in Theorem \ref{th:q-cau-cont}, we get the continued fraction expansion of the generating function of the classical Cauchy numbers (see, e.g., \cite[(8)]{Ko20},\cite[Chapter 8]{Loya}).

\subsection{$q$-Euler numbers}   

For $N\ge 0$, the $q$-hypergeometric Euler numbers $\mathcal E_{N,n}=\mathcal E_{N,n,q}$ are introduced as the generating function 
\begin{equation}  
\sum_{k=0}^\infty\mathcal E_{N,k}\frac{x^k}{k!}=\left(\sum_{n=0}^\infty\frac{[2 N]_q! x^{2 n}}{[2 N+2 n]_q!}\right)^{-1}\,.  
\label{def:q-eul}
\end{equation}  
Note that when $q\to 1$, $E_{N,n}=\mathcal E_{N,n,1}$ are hypergeometric Euler numbers \cite{KZ}. When $N=0$ and $q\to 1$, $E_n=\mathcal E_{0,n,1}$ are the classical Euler numbers because the right-hand side of (\ref{def:q-eul}) becomes $1/\cosh x$.  

Many different kinds of $q$-Euler numbers and their generalizations have been introduced (see, e.g., \cite{KimLee13}).   
However, the $q$-Euler numbers introduced in this way (\ref{def:q-eul}) have the advantage of admitting elegant determinant expressions and continued fraction expansions.   

By applying Theorem \ref{th:det-cont} with $g_j=[2 N+2 j-1]_q[2 N+2 j]_q$ and $h_j=1$ ($j\ge 1$), with $x$ being replaced by $x^2$, we can get the results about $q$-hypergeometric Euler numbers.  
The determinant of the $q$-hypergeometric Euler numbers is given as follows.  

\begin{theorem}  
For $N\ge 0$ and $n\ge 1$, we have 
$$
\mathcal E_{N,2 n}=(-1)^n(2 n)!\left|\begin{array}{ccccc}
\frac{[2 N]_q!}{[2 N+2]_q!}&1&0&&\\  
\frac{[2 N]_q!}{[2 N+4]_q!}&\frac{[2 N]_q!}{[2 N+2]_q!}&1&&\\
\vdots&\vdots&\ddots&1&0\\
\frac{[2 N]_q!}{[2 N+2 n-2]_q!}&\frac{[2 N]_q!}{[2 N+2 n-4]_q!}&\cdots&\frac{[2 N]_q!}{[2 N+2]_q!}&1\\ 
\frac{[2 N]_q!}{[2 N+2 n]_q!}&\frac{[2 N]_q!}{[2 N+2 n-2]_q!}&\cdots&\frac{[2 N]_q!}{[2 N+4]_q!}&\frac{[2 N]_q!}{[2 N+2]_q!}
\end{array}\right|\,.  
$$
\label{th-det-eul}
\end{theorem}

When $q\to 1$, Theorem \ref{th-det-eul} reduces to \cite[Proposition 2]{KZ}.  
When $N=0$, we have the determinant expression of $q$-Euler numbers $\mathcal E_{k}$.  

\begin{Cor}   
For $n\ge 1$, we have 
$$
\mathcal E_{2 n}=(-1)^n(2 n)!\left|\begin{array}{ccccc}
\frac{1}{[2]_q!}&1&0&&\\  
\frac{1}{[4]_q!}&\frac{1}{[2]_q!}&1&&\\
\vdots&\vdots&\ddots&1&0\\
\frac{1}{[2 n-2]_q!}&\frac{1}{[2 n-4]_q!}&\cdots&\frac{1}{[2]_q!}&1\\ 
\frac{1}{[2 n]_q!}&\frac{1}{[2 n-2]_q!}&\cdots&\frac{1}{[4]_q!}&\frac{1}{[2]_q!}
\end{array}\right|\,.  
$$
\end{Cor} 

Further, taking $q\to 1$ yields the determinant expression of Euler numbers $E_n$ \cite[p.~52]{Glaisher}.


\begin{Lem}  
We have 
$$
\sum_{k=0}^n\frac{[2 N]_q!\mathcal E_{N,2 k}}{(2 k)![2 N+2 n-2 k]_q!}=\begin{cases}
0&\text{if $n\ge 1$},\\
1&\text{if $n=0$}\,. 
\end{cases}
$$
\label{rel-eul}
\end{Lem}

The $q$-hypergeometric Euler numbers $\mathcal E_{N,n}$ defined in (\ref{def:q-eul}) have the following continued fraction expansion.   

\begin{theorem}  
We have 
{\small 
\begin{align*}
&\sum_{k=0}^\infty\mathcal E_{N, k}\frac{x^k}{k!}\\
&=1-\cfrac{x^2}{[2 N+1]_q[2 N+2]_q+x^2-\cfrac{[2 N+1]_q[2 N+2]_q x^2}{[2 N+3]_q[2 N+4]_q+x^2-\cfrac{[2 N+3]_q[2 N+4]_q x^2}{[2 N+5]_q[2 N+6]_q {+ x^2}-\ddots}}}\, 
\end{align*} 
} 
\label{th:q-eul-cont}
\end{theorem}

\noindent 
{\it Remark.}  
When $q\to 1$ in Theorem \ref{th:q-eul-cont}, we have \cite[Theorem 3]{Ko20}.  
\bigskip 

For $N\ge 0$, the $q$-hypergeometric Euler numbers of the second kind ($q$-complementary Euler numbers) $\mathcal{\widehat E}_{N,n}=\mathcal {\widehat E}_{N,n,q}$ are introduced as the generating function 
\begin{equation}  
\sum_{k=0}^\infty\mathcal{\widehat E}_{N,k}\frac{x^k}{k!}=\left(\sum_{n=0}^\infty\frac{[2 N+1]_q! x^{2 n}}{[2 N+2 n+1]_q!}\right)^{-1}\,.  
\label{def:q-eul2}
\end{equation}  
When $q\to 1$, $\widehat E_{N,n}=\mathcal{\widehat E}_{N,n,1}$ are hypergeometric Euler numbers of the second kind \cite{KZ}. When $N=0$ and $q\to 1$, $\widehat E_n=\mathcal{\widehat E}_{0,n,1}$ are the classical Euler numbers of the second kind (complementary Euler numbers) because the right-hand side of (\ref{def:q-eul2}) becomes $x/\sinh x$.  
Then, similarly to Theorem \ref{th-det-eul}, we have the determinant expression.  
\begin{theorem}  
For $N\ge 0$ and $n\ge 1$, we have 
$$
\mathcal{\widehat E}_{N,2 n}=(-1)^n(2 n)!\left|\begin{array}{ccccc}
\frac{[2 N+1]_q!}{[2 N+3]_q!}&1&0&&\\  
\frac{[2 N+1]_q!}{[2 N+5]_q!}&\frac{[2 N+1]_q!}{[2 N+3]_q!}&1&&\\
\vdots&\vdots&\ddots&1&0\\
\frac{[2 N+1]_q!}{[2 N+2 n-1]_q!}&\frac{[2 N+1]_q!}{[2 N+2 n-3]_q!}&\cdots&\frac{[2 N+1]_q!}{[2 N+3]_q!}&1\\ 
\frac{[2 N+1]_q!}{[2 N+2 n+1]_q!}&\frac{[2 N+1]_q!}{[2 N+2 n-1]_q!}&\cdots&\frac{[2 N+1]_q!}{[2 N+5]_q!}&\frac{[2 N+1]_q!}{[2 N+3]_q!}
\end{array}\right|\,.  
$$
\label{th-det-eul2}
\end{theorem}

When $q\to 1$, Theorem \ref{th-det-eul2} reduces to \cite[Theorem 2.1]{Ko17}.

Similarly to Theorem \ref{th:q-eul-cont}, we have the following continued fraction expansion.   

\begin{theorem}  
We have 
{\small 
\begin{align*}
&\sum_{k=0}^\infty\mathcal{\widehat E}_{N, k}\frac{x^k}{k!}\\
&=1-\cfrac{x^2}{[2 N+2]_q[2 N+3]_q+x^2-\cfrac{[2 N+2]_q[2 N+3]_q x^2}{[2 N+4]_q[2 N+5]_q+x^2-\cfrac{[2 N+4]_q[2 N+5]_q x^2}{[2 N+6]_q[2 N+7]_q {+ x^2}-\ddots}}}\, 
\end{align*} 
} 
\label{th:q-eul2-cont}
\end{theorem}

When $q\to 1$, Theorem \ref{th:q-eul2-cont} reduces to \cite[Theorem 4]{Ko20}.

\section{Finite (incomplete numbers)}   

The infinite versions of Theorem \ref{th:det-cont} can be generalized into the finite versions.  
When $r=1$, $f_n$ are known as {\it restricted} numbers. When $R\to\infty$, $f_n$ are known as {\it associated} numbers.  Both {families are examples of}    {\it incomplete} numbers (see, e.g., \cite{Ko16,Ko18,KLM16,KMS16}).  
Hence, Theorem \ref{th:det-cont} is a special case when $r=1$ and $R\to\infty$ in Theorem \ref{th:det-cont-finite}.   {In contrast to the classical infinite setting, the incomplete/truncated case naturally encodes restricted or associated numbers, and this viewpoint is one of the main structural extensions developed in this paper.}

\begin{theorem}  
Let $1\le r<R$ hold for integers $r$ and $R$. 
For three sequences $\{f_n\}_{n\ge 0}$ (with $f_0={1}$), $\{g_n\}_{n={1}}^R$, and $\{h_n\}_{n={1}}^R$, we have 
{\small 
\begin{align*} 
&\sum_{n=0}^\infty f_n x^n=\left(1+\sum_{j=r}^R\frac{h_1\cdots h_j}{g_1\cdots g_j}x^j\right)^{-1}:=\left(1+\sum_{j=r}^R\frac{H_j}{G_j}x^j\right)^{-1}\\
&=1-\cfrac{h_1\cdots h_r x^r}{g_1\cdots g_r+h_1\cdots h_r x^r-\cfrac{g_1\cdots g_r h_{r+1}x}{g_{r+1}+h_{r+1}x-\cfrac{g_{r+1}h_{r+2}x}{g_{r+2}+h_{r+2}x-{\atop\ddots-\cfrac{g_{R-1}h_R x}{g_R+h_R x}}}}}\\ 
&\Longleftrightarrow\\
&{\text{\ for\  } n\geq r, }\ \ f_n=(-1)^n\left|\begin{array}{ccc} 
\underbrace{ 
\begin{array}{ccc}
0&1&0\\
\vdots&\ddots&1\\
0&&\\
\frac{H_r}{G_r}&\ddots&\\ 
\vdots&\ddots&\\
\frac{H_R}{G_R}&&\ddots\\ 
0&\ddots&\\
\vdots&&\ddots\\
0&\cdots&0
\end{array}
}_{n-R}
&\underbrace{ 
\begin{array}{ccc}
\cdots&&\\
&&\\
\ddots&&\\
&&\\
&&\\
&&\\ 
&&\\
\ddots&&\\
&\ddots&\ddots\\
\frac{H_R}{G_R}&\cdots&\frac{H_r}{G_r}
\end{array} 
}_{R-r+1}
&\underbrace{ 
\begin{array}{ccc}
&\cdots&0\\
&&\vdots\\
&&\\
&&\\
&&\\
&&\\
&&\\
\ddots&&\vdots\\
&1&0\\
&\ddots&1\\
0&\cdots&0\\
\end{array}
}_{r-1} 
\end{array} 
\right|\\
&\Longleftrightarrow\quad \sum_{{k=\max(0,n-R)}}^{n-r}\frac{f_k H_{n-k}}{G_{n-k}}=\begin{cases}
1&\text{if $n=0$},\\
0&\text{if $n\ge 1$}
\end{cases}\\ 
&\Longleftrightarrow\\ 
&(-1)^n\left|\begin{array}{ccccc}
f_1&1&0&&\\  
f_2&f_1&1&&\\
\vdots&\vdots&\ddots&&0\\
f_{n-1}&f_{n-2}&\cdots&f_1&1\\ 
f_n&f_{n-1}&\cdots&f_2&f_1
\end{array}\right|=\begin{cases}
\frac{H_n}{G_n}&\text{if $r\le n\le R$},\\
0&\text{otherwise}
\end{cases}\\
&\Longleftrightarrow\\ 
&f_n=\sum_{r t_r+(r+1)t_{r+1}+\cdots+R t_R=n}\binom{t_r+\cdots+t_R}{t_r,\dots,t_R}(-1)^{t_r+\cdots+t_R}\\
&\qquad\qquad\times\left(\frac{H_r}{G_r}\right)^{t_r}\left(\frac{H_{r+1}}{G_{r+1}}\right)^{t_{r+1}}\cdots\left(\frac{H_R}{G_R}\right)^{t_R}\\
&\phantom{f_n}=\sum_{k=1}^n(-1)^k\sum_{i_1+\cdots+i_k=n\atop r\le i_1,\dots,i_k\le R}\frac{H_{i_1}}{G_{i_1}}\cdots\frac{H_{i_k}}{G_{i_k}}\\
&\Longleftrightarrow\\ 
&\sum_{t_1+2 t_2+\cdots+n t_n=n}\binom{t_1+\cdots+t_n}{t_1,\dots,t_n}(-1)^{t_1+\cdots+t_n}f_1^{t_1}f_{2}^{t_{2}}\cdots f_n^{t_n}\\
&=\sum_{k=1}^n(-1)^k\sum_{i_1+\cdots+i_k=n\atop i_1,\dots,i_k\ge 1}f_{i_1}\cdots f_{i_k}=\begin{cases}
\frac{H_n}{G_n}&\text{if $r\le n\le R$},\\
0&\text{otherwise}\,.
\end{cases}   
\end{align*}
} 
\label{th:det-cont-finite}
\end{theorem} 
\begin{proof}  
For finite versions, we also apply \cite[Lemmas 1--8]{Ko20b}.
\end{proof}

\subsection{Examples}  

{
Let $r=2$ and $R=4$. We illustrate Theorem \ref{th:det-cont-finite} in the case where $n=6$.
First, consider the formula
\[
f_n=\sum_{k=1}^n(-1)^k
\sum_{i_1+\cdots+i_k=n \atop r\le i_1,\dots,i_k\le R}
\frac{H_{i_1}}{G_{i_1}}\cdots\frac{H_{i_k}}{G_{i_k}}.
\]
For $r=2$, $R=4$, and $n=6$, we must sum over all compositions of $6$ with parts in $\{2,3,4\}$. These are
\[
6=2+2+2,\qquad 6=3+3,\qquad 6=2+4,\qquad 6=4+2.
\]
Hence
\[
\begin{aligned}
f_6
&= -\left(\frac{H_2}{G_2}\right)^3
   +\left(\frac{H_3}{G_3}\right)^2
   +\left(\frac{H_2}{G_2}\right)\left(\frac{H_4}{G_4}\right)
   +\left(\frac{H_4}{G_4}\right)\left(\frac{H_2}{G_2}\right) \\
&= -\left(\frac{H_2}{G_2}\right)^3
   +\left(\frac{H_3}{G_3}\right)^2
   +2\left(\frac{H_2}{G_2}\right)\left(\frac{H_4}{G_4}\right).
\end{aligned}
\]
Next, consider the last relation in Theorem \ref{th:det-cont-finite}:
\[
\sum_{k=1}^n(-1)^k
\sum_{i_1+\cdots+i_k=n \atop i_1,\dots,i_k\ge 1}
f_{i_1}\cdots f_{i_k}
=
\begin{cases}
\dfrac{H_n}{G_n} & \text{if } r\le n\le R,\\[1mm]
0 & \text{otherwise}.
\end{cases}
\]
Since $6>R=4$, its right-hand side is $0$. Thus
\[
\sum_{k=1}^6(-1)^k
\sum_{i_1+\cdots+i_k=6 \atop i_1,\dots,i_k\ge 1}
f_{i_1}\cdots f_{i_k}=0.
\]
Now $f_1=0$, so only the compositions
\[
6,\qquad 3+3,\qquad 2+4,\qquad 4+2,\qquad 2+2+2
\]
contribute. Therefore
\[
-f_6+f_3^2+2f_2f_4-f_2^3=0,
\]
that is,
\[
-f_6-f_2^3+f_3^2+2f_2f_4=0.
\]
Thus, for $n=6$, Theorem \ref{th:det-cont-finite} gives both an explicit expression for $f_6$ in terms of $H_j/G_j$ and, independently, a polynomial relation among the $f_j$.}


\section{Applications to Lehmer-Euler numbers}   

In 1935, D. H. Lehmer \cite{Lehmer} introduced and investigated generalized Euler numbers $W_n$, defined by the generating function 
\begin{equation}  
\sum_{n=0}^\infty W_n\frac{t^n}{n!}=\frac{3}{e^t+e^{\omega t}+e^{\omega^2 t}}\,,
\label{def:lehmer-euler}
\end{equation}
where $\omega=(-1+\sqrt{-3})/2$ is the primitive cubic root of unity. In \cite{BK19}, more general numbers were studied in terms of determinants, which involve Bernoulli, Euler, and Lehmer's generalized Euler numbers. In \cite[Propositions 3,4]{KL25}, several expressions of incomplete versions of Lehmer-Euler numbers were given, except for continued fractions. Here, we obtain results for the incomplete version of $q$-Lehmer-Euler numbers $\mathcal W_n^\ast:=W_{n,q}^\ast$ by applying Theorem \ref{th:det-cont-finite} with $g_j=[3 N+3 j-2]_q[3 N+3 j-1]_q[3 N+3 j]_q$ and $h_j=1$ ($j\ge 1$), with $x$ being replaced by $x^3$.   
\begin{theorem}  
{Let $N\in\mathbb N$.} For $1\le r<R$, we have 
{\scriptsize 
\begin{align*} 
&\sum_{n=0}^\infty W_n^\ast\frac{x^n}{n!}=\sum_{n=0}^\infty W_{3 n}^\ast\frac{x^{3 n}}{(3 n)!}:=\left(1+\sum_{j=r}^R\frac{1}{[3 N+3 j]_q!}x^{3 j}\right)^{-1}\\
&=1-\cfrac{x^{3 r}}{[3 N+3 r]_q!+x^{3 r}-\cfrac{[3 N+3 r]_q! x^3}{\prod_{j=1}^{3}[3 N+3 r+j]_q+x^3-\cfrac{\prod_{j=1}^{3}[3 N+3 r+j]_q x^3}{\prod_{j=4}^{6}[3 N+3 r+j]_q+x^3-{\atop\ddots-\cfrac{\prod_{j=-5}^{{-3}}[3 N+3 R+j]_q x^3}{\prod_{j=-2}^{0}[3 N+3 R+j]_q+x^3}}}}}\, 
\end{align*}
} 
{\small  
\begin{align*} 
&W_{3 n}^\ast=
(-1)^n(3 n)!\left|\begin{array}{ccc} 
\underbrace{ 
\begin{array}{ccc}
0&1&0\\
\vdots&\ddots&1\\
0&&\\
\frac{1}{[3 N+3 r]_q!}&\ddots&\\ 
\vdots&\ddots&\\
\frac{1}{[3 N+3 R]_q!}&&\ddots\\ 
0&\ddots&\\
\vdots&&\ddots\\
0&\cdots&0
\end{array}
}_{n-R}
&\underbrace{ 
\begin{array}{ccc}
\cdots&&\\
&&\\
\ddots&&\\
&&\\
&&\\
&&\\ 
&&\\
\ddots&&\\
&\ddots&\ddots\\
\frac{1}{[3 N+3 R]_q!}&\cdots&\frac{1}{[3 N+3 r]_q!}
\end{array} 
}_{R-r+1}
&\underbrace{ 
\begin{array}{ccc}
&\cdots&0\\
&&\vdots\\
&&\\
&&\\
&&\\
&&\\
&&\\
\ddots&&\vdots\\
&1&0\\
&\ddots&1\\
0&\cdots&0\\
\end{array}
}_{r-1} 
\end{array} 
\right|\end{align*}
} 
\begin{align*}
\sum_{{k=\max(0,n-R)}}^{n-r}\frac{W_{3 k}^\ast}{(3 k)![3 N+3 n-3 k]_q!}=\begin{cases}
1&\text{if $n=0$},\\
0&\text{if $n\ge 1$}\,. 
\end{cases}
\end{align*}
\begin{align*} 
&(-1)^n\left|\begin{array}{ccccc}
\frac{W_3^\ast}{3!}&1&0&&\\  
\frac{W_6^\ast}{{6!}}&\frac{W_3^\ast}{3!}&1&&\\
\vdots&\vdots&\ddots&&0\\
\frac{W_{3 n-3}^\ast}{(3 n-3)!}&\frac{W_{3 n-6}^\ast}{(3 n-6)!}&\cdots&\frac{W_6^\ast}{6!}&1\\ 
\frac{W_{3 n}^\ast}{(3 n)!}&\frac{W_{3 n-3}^\ast}{(3 n-3)!}&\cdots&\frac{W_6^\ast}{6!}&\frac{W_3^\ast}{3!}
\end{array}\right|=\begin{cases}
\frac{1}{[3 N+3 n]_q!}&\text{if $r\le n\le R$},\\
0&\text{otherwise}\,.
\end{cases}
\end{align*}
\begin{align*}
&W_{3 n}^\ast=(3 n)!\sum_{r t_r+(r+1)t_{r+1}+\cdots+R t_R=n}\binom{t_r+\cdots+t_R}{t_r,\dots,t_R}(-1)^{t_r+\cdots+t_R}\\
&\qquad\qquad\times\left(\frac{1}{[3 N+3 r]_q!}\right)^{t_r}\left(\frac{1}{[3 N+3 r+3]_q!}\right)^{t_{r+1}}\cdots\left(\frac{1}{[3 N+3 R]_q!}\right)^{t_R}\\
&\phantom{W_{3 n}^\ast}=(3 n)!\sum_{k=1}^n(-1)^k\sum_{i_1+\cdots+i_k=n\atop r\le i_1,\dots,i_k\le R}\frac{1}{[3 N+3 i_1]_q!}\cdots\frac{1}{[3 N+3 i_k]_q!}\,.
\end{align*}
\begin{align*}
&\sum_{t_1+2 t_2+\cdots+n t_n=n}\binom{t_1+\cdots+t_n}{t_1,\dots,t_n}(-1)^{t_1+\cdots+t_n}\left(\frac{W_3^\ast}{3!}\right)^{t_1}\left(\frac{W_6^\ast}{6!}\right)^{t_{2}}\cdots\left(\frac{W_{3 n}^\ast}{(3 n)!}\right)^{t_n}\\
&=\sum_{k=1}^n(-1)^k\sum_{i_1+\cdots+i_k=n\atop i_1,\dots,i_k\ge 1}\frac{W_{3 i_1}^\ast\cdots W_{3 i_k}^\ast}{(3 i_1)!\cdots(3 i_k)!}=\begin{cases}
\frac{1}{[3 N+3 n]_q!}&\text{if $r\le n\le R$},\\
0&\text{otherwise}\,.
\end{cases}   
\end{align*}
\label{cont-lehmer-euler}  
\end{theorem}

\section{Applications to determinantal harmonic numbers}  

In \cite{KP19}, for $r\ge 0$, the {\it determinantal hyperharmonic numbers} $h_n^{(r)}$ {\it  of order $r$} are defined by the generating function 
\begin{equation}  
\frac{1}{1-\log(1+x)/(1+x)^r}=\sum_{n=0}^\infty h_n^{(r)}x^n\quad(|x|<1)\,,  
\label{def:dethyperharmo} 
\end{equation} 
while the generating function of hyperharmonic numbers $H_n^{(r)}$ of order $r$ is given by 
\begin{equation}  
-\frac{\log(1-z)}{(1-z)^r}=\sum_{n=1}^\infty H_n^{(r)}z^n\,. 
\end{equation} 
Then the numbers $H_n^{(r)}$ are produced iteratively by 
\begin{equation}  
H_n^{(r)}=\sum_{k=1}^n H_k^{(r-1)}\quad(r\ge 1)\quad\hbox{with}\quad H_n^{(0)}=\frac{1}{n}\,.
\label{def:hyperharmo} 
\end{equation}
It is known by \cite{CG96} that $H_n^{(r)}$ can be expressed in terms of binomial coefficients and harmonic numbers with the formula 
$$
H_n^{(r)}=\binom{n+r-1}{r-1}(H_{n+r-1}-H_{r-1})\,.  
$$

In order to generalize the previous results, we introduce the partial summation of $\log(1+z)$ by 
\begin{equation}  
F_m(z)=z-\frac{z^2}{2}+\cdots+\frac{(-1)^{m-1}z^m}{m}\,. 
\label{def:part-sum-log} 
\end{equation} 
The {\it restricted determinantal hyperharmonic numbers} $h_{n,\le\ell}^{(r)}$ {\it  of order $r$} are defined by 
\begin{equation}  
\sum_{n=0}^\infty h_{n,{\le\ell}}^{(r)}x^n=\left(1+\sum_{j=1}^\ell(-1)^j H_j^{(r)}x^j\right)^{-1}\quad(|x|<1)\,.    
\end{equation}  
Then we have the following.   

\begin{Prop}  
For $1\le\ell\le n$, we have 
$$
h_{n,\le\ell}^{(r)}=\left|\begin{array}{cc} 
\underbrace{ 
\begin{array}{ccc}
H_1^{(r)}&1&0\\
\vdots&&\ddots\\
H_{\ell}^{(r)}&&\\
0&\ddots&\\ 
&&\ddots\\
&&0
\end{array}
}_{n-\ell} 
\underbrace{ 
\begin{array}{ccc}
&&\\
\ddots&&\\
&\ddots&\\
&\ddots&0\\
&&1\\
H_{\ell}^{(r)}&\cdots&H_1^{(r)}
\end{array}
}_{\ell} 
\end{array} 
\right|\,. 
$$
\label{prp:res-det-hyperharmo} 
\end{Prop}  

\noindent 
{\it Remark.}  
When $\ell\ge n$, we have $h_{n,\le\ell}^{(r)}=h_n^{(r)}$.

In \cite{KP19}, for $m,n,r\ge 0$,  the {\it shifted determinantal hyperharmonic numbers} $h_{n,m}^{(r)}$ are defined by  
\begin{equation}  
\sum_{n=0}^\infty h_{n,m}^{(r)}x^n=\left(1+\frac{F_{m-1}(x)-\log(1+x)}{(-x)^{m-1}(1+x)^r}-x\sum_{j=1}^r\frac{H_{m-1}^{(j)}}{(1+x)^{r-j+1}}\right)^{-1}\,. 
\label{def:shift-det-hyperharmo} 
\end{equation}
Then we have the determinant expressions 
$$
h_{n,m}^{(r)}=\left|\begin{array}{ccccc}
H_m^{(r)}&1&0&&\\  
H_{m+1}^{(r)}&H_m^{(r)}&1&&\\
\vdots&\vdots&\ddots&1&0\\
H_{m+n-2}^{(r)}&H_{m+n-3}^{(r)}&\cdots&H_m^{(r)}&1\\ 
H_{m+n-1}^{(r)}&H_{m+n-2}^{(r)}&\cdots&H_{m+1}^{(r)}&H_{m}^{(r)}
\end{array}\right|
$$ 
and 
$$
H_{m+n-1}^{(r)}=\left|\begin{array}{ccccc}
h_{1,m}^{(r)}&1&0&&\\  
h_{2,m}^{(r)}&h_{1,m}^{(r)}&1&&\\
\vdots&\vdots&\ddots&1&0\\
h_{n-1,m}^{(r)}&h_{n-2,m}^{(r)}&\cdots&h_{1,m}^{(r)}&1\\ 
h_{n,m}^{(r)}&h_{n-1,m}^{(r)}&\cdots&h_{2,m}^{(r)}&h_{1,m}^{(r)}
\end{array}\right|\,. 
$$ 

One of the main aims of this section is to give a determinantal expression, as shown in Proposition~\ref{th:det-shift-rest-hyperharmo}. 
For convenience, we set 
\begin{equation}  
\mathfrak L_m^{(r)}(z)=\sum_{k=1}^m\frac{z^k}{k^r}\,, 
\label{def:part-polylog}
\end{equation}  
so when $m\to\infty$, 
$$
\mathfrak L_\infty^{(r)}(z)=\sum_{k=1}^\infty\frac{z^k}{k^r}={\rm Li}_r(z)
$$
is the polylogarithm function.   
For $m,n,r\ge 0$ and $\ell\ge 1$,  the {\it restricted shifted determinantal hyperharmonic numbers} $h_{n,m,\le\ell}^{(r)}$ are defined by  
\begin{align}  
&\sum_{n=0}^\infty h_{n,m,\le\ell}^{(r)}x^n
=\left(1+\frac{H_{m-1}^{(r)}x\bigl(1-(-x)^\ell\bigr)}{1+x}\right.\notag\\
&\qquad\left.+\frac{\mathfrak L_{\ell+m-1}^{(r)}(-x)-\mathfrak L_{m-1}^{({r})}(-x)}{(1+x)(-x)^{m-1}}-\frac{(-x)^{\ell+1}(H_{\ell+m-1}^{(r)}-H_{m-1}^{(r)})}{1+x}\right)^{-1}\,. 
\label{def:rist-shift-det-hyperharmo} 
\end{align}
Notice that when $\ell\to\infty$, this expression reduces to that of (\ref{def:shift-det-hyperharmo}). Hence, $h_{n,m}^{(r)}=h_{n,m,\le\infty}^{(r)}$. 

\begin{Prop}  
For $1\le\ell\le n$, $m\ge 1$, and $r\ge 0$, we have 
$$
h_{n,m,\le\ell}^{(r)}=\left|\begin{array}{cc} 
\underbrace{ 
\begin{array}{ccc}
H_m^{(r)}&1&0\\
\vdots&&\ddots\\
H_{m+\ell-1}^{(r)}&&\\
0&\ddots&\\ 
&&\ddots\\
&&0
\end{array}
}_{n-\ell} 
\underbrace{ 
\begin{array}{ccc}
&&\\
\ddots&&\\
&\ddots&\\
&\ddots&0\\
&&1\\
H_{m+\ell-1}^{(r)}&\cdots&H_m^{(r)}
\end{array}
}_{\ell} 
\end{array} 
\right|\,. 
$$
\label{th:det-shift-rest-hyperharmo}
\end{Prop}

A continued fraction expansion of the generating function of $h_n^{(r)}$ can be obtained indirectly because it is difficult to determine $g_j$ and $h_j$ in Theorem \ref{th:det-cont} or in Theorem \ref{th:det-cont-finite} from (\ref{def:dethyperharmo}).  See also \cite[Corollary 8]{Ko20}.  The structure of the hyperharmonic numbers is similar, {see}  
(\ref{def:hyperharmo}). Note that a continued fraction expansion of higher-order harmonic numbers $\sum_{j=1}^n(1/j^m)$ is given in \cite[Corollary 7]{Ko20}. 

\begin{theorem}
We have 
\begin{align*}
&\sum_{n=0}^\infty h_n^{(r)}x^n=\frac{1}{1-\log(1+x)/(1+x)^r}\\
&=\cfrac{1}{1-\cfrac{x}{(1+x)^r+\cfrac{x(1+x)^r}{2-x+\cfrac{2^2 x}{3-2 x+\cfrac{3^2 x}{4-3 x-\ddots}}}}} 
\end{align*}
and 
\begin{align*}
&\sum_{n={1}}^\infty H_n^{(r)}x^n=-\frac{\log(1-x)}{(1-x)^r}\\
&=\cfrac{x}{(1-x)^r-\cfrac{x(1-x)^r}{2+x-\cfrac{2^2 x}{3+2 x-\cfrac{3^2 x}{4+3 x+\cfrac{4^2 x}{5+4 x-\ddots}}}}}\,  
\end{align*}
\label{th:cont-detharmo}
\end{theorem}  
\begin{proof}   
By using the continued fraction expansion of Cauchy numbers in (\ref{eq:q-cau-cont}) with $q\to 1$, from (\ref{def:dethyperharmo}), we have 
\begin{align*}
&\sum_{n=0}^\infty h_n^{(r)}x^n=\frac{1}{1-\log(1+x)/(1+x)^r}\\
&=\cfrac{1}{1-\cfrac{x}{(1+x)^r\cfrac{x}{\log(1+x)}}}\\
&=\cfrac{1}{1-\cfrac{x}{(1+x)^r+\cfrac{x(1+x)^r}{2-x+\cfrac{2^2 x}{3-2 x+\cfrac{3^2 x}{4-3 x-\ddots}}}}}\,  
\end{align*}
\end{proof}

\noindent 
{\it Remark.}  
As we use the continued fraction of Cauchy numbers, it holds that 
$$
\sum_{n=0}^\infty h_n^{(r)}x^n\equiv \cfrac{1}{1-\cfrac{x}{(1+x)^r+\cfrac{x(1+x)^r}{2-x+\cfrac{2^2 x}{\ddots{\atop+\cfrac{(m-1)^2 x}{m-(m-1)x}}}}}}
\pmod{x^{m+1}}\,.  
$$ 
That is, this continued fraction can yield the terms $h_n^{(r)}$ ($n=1,2,\dots,m$) correctly.

\section{Concluding remarks}   

In this paper, we focus on the relations between the determinant expressions and the continued fraction expansions. In particular, for continued fractions of the type (\ref{j-cont-cauchy}), we obtained several determinant formulas and coefficient identities.  
However, with type of (\ref{t-cont-cauchy}) and its modification 
$$
\sum_{n=0}^\infty c_n\frac{x^n}{n!}=1+\cfrac{x}{2+\cfrac{x}{3+\cfrac{2 x}{2+\cfrac{2 x}{5+\cfrac{3 x}{2+{\atop\ddots}}}}}}
$$ 
({\it cf.} \cite{Wolfram-Log}), as well as other types, few relations are currently known, though one generalization is given in \cite[Theorem 9]{Ko20}. Because of the complicated structures, nontrivial relations have been derived \cite[Theorem 1]{DK19}.   

On the other hand, there also exist modifications of Trudi-type determinant expressions. See, e.g., \cite[Lemma 4]{KS25} and \cite[Lemma 1]{Ko25}. 
It would be interesting to find out whether or not there are some continued fraction expansions or further identities.

\section*{Author contributions}
Takao Komatsu: Writing -- original draft; Nikita Kalinin: Writing -- review and editing; All authors have read and agreed to the published version of the manuscript.

\section*{Acknowledgments}
The authors would like to thank the organizer of the DART2025Z conference for providing a stimulating and productive research environment.
{The authors thank the anonymous referees for their careful reading of the manuscript and helpful comments and suggestions.}  
T.K. was partly supported by JSPS KAKENHI Grant Number 24K22835.

\section*{Conflict of interest}
Prof. Komatsu is the Guest Editor of special issue "Number theory and its surroundings--New studies and developments
" for AIMS Mathematics. Prof. Komatsu was not involved in the editorial review and the decision to publish this article. All authors declare no conflicts of interest in this paper.

\end{document}